\documentclass[11pt]{article}
\usepackage{amsmath}
\usepackage{amsfonts}
\usepackage{mathtools}
\usepackage{enumerate}
\usepackage{amsthm}
\usepackage[textwidth=5.6in,textheight=8.0625in]{geometry}
\newtheorem{theorem}{Theorem}[section]
\newtheorem{corollary}[theorem]{Corollary}
\newtheorem{lemma}[theorem]{Lemma}

\newtheorem{remark}[theorem]{Remark}
\newtheorem{definition}[theorem]{Definition}
\newtheorem{example}[theorem]{Example}
\newtheorem{condition}[theorem]{Condition}

\numberwithin{equation}{section}
\def\a{\alpha}				% discount factor
\def\e{\epsilon}				% infinitesimal
\def\bS{\mathbf{S}}			% bold S
\def\Sa{S_{\a}}				% (s_a,S_a)-policy
\def\sa{s_{\a}}				% (s_a,S_a)-policy
\def\va{v_{\a}}				% Discounted value function
\def\Ga{G_{\a}}				% Discounted G_\a
\def\E{\mathbb{E}}			% expectation
\def\PR{\mathbb{P}}			% probability
\def\R{\mathbb{R}}			% field R
\def\Z{\mathbb{Z}}			% field Z
\def\N{\mathbb{N}}			% field N
\def\U{\mathbb{U}}			% field U
\def\Rp{\R_{+}}				% non-negative real numbers
\def\X{\mathbb{X}}			% field X
\def\A{\mathbb{A}}			% field A
\def\argmin{\arg\!\min}		% argmin
\def\d{\delta}				% delta
\title{
Structure of Optimal Policies to Periodic-Review Inventory Models with Convex Costs and Backorders for all Values of Discount Factors
}
\author{
Eugene~A.~Feinberg, Yan~Liang \\
\small
\textit{Department of Applied Mathematics and Statistics} \\
\small
\textit{Stony Brook University, Stony Brook, NY 11794} \\
\small
\textit{eugene.feinberg@stonybrook.edu, yan.liang@stonybrook.edu}
}
\date{}
\begin{document}
\maketitle
%
%	abstract
%
\begin{abstract}
\textit{
	This paper describes the structure of  optimal policies for discounted periodic-review single-commodity
	total-cost inventory control problems with fixed ordering costs for finite and infinite horizons.
There are known conditions in the literature for optimality of $(s_t,S_t)$ policies
	for finite-horizon problems and the optimality of $(s,S)$ policies for infinite-horizon
	problems. The results of this paper cover the  situation, when such assumption may not hold. This paper describes a parameter, which, together with the value of the discount factor and the horizon length,
defines the structure of an optimal policy. For the infinite horizon, depending on the values of this parameter and the discount factor,  an optimal policy either is an $(s,S)$ policy or  never orders inventory.  For a finite horizon, depending on the values of this parameter, the discount factor, and the horizon length,  there are three possible structures of an optimal policy: (i) it is an $(s_t,S_t)$ policy, (ii) it is an $(s_t,S_t)$ policy at earlier stages and then does not order inventory, or (iii) it never orders inventory. The paper also establishes continuity of optimal value functions and describes alternative optimal actions at states $s_t$ and $s.$
}
\end{abstract}
%
%	keywords
%
\section{Introduction}
\label{sec:introduction}

It is well-known that, for the classic periodic-review single-commodity inventory
control problems with fixed ordering costs, $(s,S)$ policies are optimal for the expected total cost criterion under certain
conditions on cost functions.  These policies  order up to the level $S,$ when the inventory level
is less than $s,$ and do not order otherwise.  This paper investigates the general situations, when $(s,S)$ policies may not be optimal.

%This paper describes the structure
%of optimal policies for inventory control problems with total cost criterion when
%above-mentioned conditions may not hold.  It also provides the results on the
%continuity of optimal value functions.

Systematic studies of inventory control problems started with the papers by
Arrow et al. \cite{AHM} and Dvoretzky et al. \cite{DKW}. Most of the earlier
results are surveyed in the books by Bensoussan~\cite{ben}, Beyer et al.~\cite{BCST}, Porteus \cite{PORTEUS1990}, Simchi-Levi et al.~\cite{SCB}, and Zipkin \cite{ZIPKIN2000}, see also 
Katehakis et al.~\cite{KMS16} and Shi et al.~\cite{SKM13} for recent results for continuous review models.

Recently developed general optimality conditions for discrete-time Markov Decision Processes (MDPs)  
applicable to inventory control problem are
discussed in the tutorial by Feinberg \cite{Ftut}. Here, we mention just a few directly relevant references.
Scarf \cite{SCARF1960} introduced the concept of $K$-convexity to prove the optimality of
$(s,S)$ policies for finite-horizon problems with continuous demand.
Zabel \cite{za62} indicated some gaps in Scarf \cite{SCARF1960} and corrected them.
Iglehart \cite{IGLEHART1963} extended the results in
Scarf \cite{SCARF1960} to infinite-horizon problems with continuous demand.
Veinott and Wagner \cite{VEINOTT1965} proved the optimality
of $(s,S)$ policies for both finite-horizon and infinite-horizon problems with discrete
demand. Zheng \cite{zh91} provided an alternative proof for discrete demand.
Beyer and Sethi \cite{BS99} completed the missing proofs in Iglehart
\cite{IGLEHART1963} and Veinott and Wagner \cite{VEINOTT1965}. In general, $(s,S)$ policies may not be optimal. To ensure the
optimality of $(s,S)$ policies, the additional assumption on backordering cost function (see Condition~\ref{cond:h(x)} below) is
 used in many papers including Iglehart \cite{IGLEHART1963} and  Veinott and Wagner \cite{VEINOTT1965}. Relevant assumptions are used in Sch\"al \cite{sch75}, Heyman and Sobel \cite{HEYMAN1984}, Bertsekas \cite{BERTSEKAS2000}, Chen and Simchi-Levi \cite{CHEN2004a,CHEN2004b}, Huh and Janakiraman \cite{hj08}, and Huh et al. \cite{HUH2011}.  As shown by Veinott and Wagner \cite{VEINOTT1965} for problems with discrete demand and Feinberg and Lewis \cite{FEINBERG2015} for an arbitrary distributed demand, such assumptions are not needed for an infinite-horizon problem, when the discount factor is close to $1.$

%In order to describe the optimal actions at state $s$ for an $(s,S)$ policy, the continuity of the value functions is needed. Simchi-Levi et al. \cite[Theorem 8.3.4, p. 126]{SCB} %consider an inventory model with linear holding costs and continuous demand, and prove that the finite-horizon value functions are continuous. Bensoussan \cite[Theorem 9.11, p. %118]{ben} considers the same model with an additional assumption of the holding costs and provides the results on the continuity of infinite-horizon value functions.

For problems with linear holding costs, according to Simchi-Levi et al. \cite[Theorem 8.3.4, p. 126]{SCB},  finite-horizon undiscounted value functions are continuous and, according to Bensoussan \cite[Theorem 9.11, p. 118]{ben}, infinite-horizon discounted value functions are continuous. These continuity properties obviously hold, if the amounts of stored inventory are limited only to integer values, and they are nontrivial if the amounts of stored inventory are modeled by real-valued numbers. General results on MDPs  state only lower semi-continuity of discounted value functions; see Feinberg et al.~\cite[Theorem 2]{FEINBERG2012}.

This paper studies the structure of optimal policies without the assumption on backordering costs mentioned above. We describe a parameter, which, together with the value of the discount factor and the horizon length,
defines the structure of an optimal policy.  For a finite horizon, depending on the values of this parameter, the discount factor, and the horizon length,  there are three possible structures of an optimal policy: (i) it is an $(s_t,S_t)$ policy, (ii) it is an $(s_t,S_t)$ policy at earlier stages and then does not order inventory, or (iii) it never orders inventory.
For the infinite horizon, depending on the values of this parameter and the discount factor,  an optimal policy either is an $(s,S)$ policy or  never orders inventory.
  This paper also establishes continuity of optimal discounted value functions for finite and infinite-horizon problems.   The continuity of values functions is used to prove that, if the amount of stored inventory is modeled by real numbers, then ordering up to the levels $S_t$ and $S$  are also optimal actions at states $s_t$ and $S_t$ respectively for discounted  finite and infinite-horizon problems; see Corollary~\ref{thm:ordering at sa} below.

The rest of the paper is organized in the following way.  Section \ref{sec:model definition} introduces
the classic  stochastic periodic-review single-commodity inventory control problems with fixed ordering costs.
Section \ref{sec:inventory control} presents the known results on the optimality of $(s,S)$ policies. Section \ref{sec:structure general results} describes the structure
of optimal policies for finite-horizon and infinite-horizon problems for
all possible values of discount factors. Section \ref{sec:continuity} establishes continuity of value functions and describes the optimal actions at states $s_t$ and $s.$

\section{Model definition}
\label{sec:model definition}

Let $\R$ denote the real line, $\Z$ denote the set of all integers, $\Rp:=[0,+\infty)$ and
$\N_0 = \{0,1,2,\ldots \}.$ Consider the classic stochastic
periodic-review inventory control problem with fixed ordering cost and general demand.
At times $t=0,1,\ldots,$ a decision-maker views the current inventory of a single commodity
and makes an ordering decision. Assuming zero lead times, the products are immediately
available to meet demand. Demand is then realized, the decision-maker views the remaining
inventory, and the process continues. The unmet demand is backlogged and the cost of
inventory held or backlogged (negative inventory) is modeled as a convex function. The
demand and the order quantity are assumed to be non-negative. The state and action spaces
are either (i) $\X = \R$ and $\A = \Rp,$ or (ii) $\X=\Z$ and $\A = \N_0.$
The inventory control problem is defined by the following parameters.
\begin{enumerate}
    \item $K> 0$ is a fixed ordering cost;
    \item $\bar{c}>0$ is the per unit ordering cost;
    \item $h(\cdot)$ is the holding/backordering cost per period, which is assumed to be
    a convex real-valued function on $\X$ such that $h(x)\to\infty$ as $|x|\to\infty;$
    without loss of generality, consider $h(\cdot)$ to be non-negative;
    \item $\{D_t,t=1,2,\ldots\}$ is a sequence of i.i.d. non-negative finite
    random variables representing the demand at periods $0,1,\ldots\ .$
    We assume that $\E[D] < +\infty$ and $\PR(D>0)>0,$
    where $D$ is a random variable with the same distribution as $D_1;$ \label{enum:demand}
    \item $\a \geq 0$ is the discount factor for finite-horizon problems, and
    $\a\in [0,1)$ for infinite-horizon problems.
\end{enumerate}
Without loss of generality, assume that $h(0) = 0.$ The assumption $\PR(D>0)>0$
avoids the trivial case when there is no demand. Define  $\bS_0:=0$ and
\begin{align}
	\bS_t := \sum_{j=1}^t D_j,\qquad\qquad\qquad t=1,2,\ldots\ .
	\label{eqn:sum of demand}
\end{align}
 Then $\E[\bS_t] = t\E[D] < +\infty$
for all $t=0,1,\ldots\ .$

Define the following function for all $y,z\in\X$ such that $y\neq z,$
\begin{align}\label{eqn:H(y,z)}
	H(y,z) := \frac{h(y)-h(z)}{y-z} .
\end{align}
 The convexity of $h$ implies that
$H(y,z)$ is non-decreasing in $y$ on $\X\setminus \{z\}$ for all $z\in\X$ and non-decreasing
in $z$ on $\X\setminus \{y\}$ for all $y\in\X;$ see  Hiriart-Urruty and Lemar\'{e}chal \cite[Proposition 1.1.4 on p. 4]{hl1993}.

Since $\frac{h(x)}{x} = H(x,0)$ is a non-decreasing function on $\X\setminus \{0\},$ then
consider the limit
\begin{align}
	k_h:= - \lim_{x\to -\infty} \frac{h(x)}{x} .
	\label{eqn:limit h(x)}
\end{align}	
Since $h(x)\to\infty$ as $x\to-\infty,$ then there exists $x^* < 0$ such that $h(x^*)>0.$
Therefore, $H(x^*,0)<0.$ Thus $0< k_h \leq +\infty.$

The dynamics of the system is defined by the equations
\begin{displaymath}
	x_{t+1} = x_t + a_t - D_{t+1}, \quad t=0,1,2,\ldots ,
\end{displaymath}
where $x_t$ and $a_t$ denote the current inventory level and the ordered amount at period $t$ respectively. If an action $a$ is chosen at state $x$ then the following cost is collected,
\begin{align}
	c(x,a) = K I_{\{ a > 0 \}} + \bar{c}a + \E[h(x+a-D)],
	\qquad (x,a)\in \X\times\A ,
	\label{inventory-control:cost function}
\end{align}
where $I_{\{a>0\}} $ is an indicator of the event $\{a>0\}.$

Let $H_t=(\X\times\A)^t\times\X$ be the sets of histories up to periods $t=0,1,\ldots\ .$ A
(randomized) decision rule at period $t=0,1,\ldots$ is a regular transition probability
$\pi_t : H_t\to \A$; that is, (i) $\pi_t(\cdot|h_t)$ is a probability distribution on $\A,$
where $h_t=(x_0,a_0,x_1,\ldots,a_{t-1},x_t),$ and (ii) for any measurable subset
$B\subset \A,$ the function $\pi_t(B|\cdot)$ is measurable on $H_t.$
A policy $\pi$ is a sequence $(\pi_0,\pi_1,\ldots)$ of decision rules.
Moreover, $\pi$ is called non-randomized if each probability measure $\pi_t(\cdot|h_t)$ is
concentrated at one point. A non-randomized policy is called Markov if all decisions depend
only on the current state and time. A Markov policy is called stationary if all decisions
depend only on the current state.

For a finite horizon $N=0,1,\ldots$ and a discount factor $\a \geq 0,$
define the expected total discounted costs
\begin{align}
	v_{N,\a}^{\pi} (x) := \E_{x}^{\pi}[\sum_{t=0}^{N-1}\a^t c(x_t,a_t)] .
	\label{eqn:finite costs}
\end{align}
When $N=+\infty$ and $\a\in[0,1),$
\eqref{eqn:finite costs} defines the infinite horizon expected total
discounted cost of $\pi$ denoted by $v_{\a}^{\pi}(x)$ instead of $v_{+\infty,\a}^{\pi} (x).$
%For each function $V^{\pi}(x)=v_{N,\a}^{\pi} (x),v_{\a}^{\pi}(x),$
Define the optimal values
\begin{displaymath}
	v_{N,\a} (x) = \inf_{\pi\in\Pi} v_{N,\a}^{\pi}(x) , \quad \text{and} \quad
	v_{\a} (x) = \inf_{\pi\in\Pi} v_{\a}^{\pi}(x)  ,
\end{displaymath}
where $\Pi$ is the set of all policies. A policy $\pi$ is called optimal for the respective
criterion if $v_{N,\a}^{\pi} (x) = v_{N,\a}(x)$ or $v_{\a}^{\pi} (x) = v_{\a}(x)$
for all $x\in\X.$

\section{Optimality of $(s,S)$ policies}
\label{sec:inventory control}

It is known that optimal $(s,S)$ policies may not exists. This section considers the known
sufficient condition for the optimality of $(s_{t,\a},S_{t,\a})$ and $(\sa,\Sa)$ policies
for discounted problems.

The  value functions for the inventory control problem defined in
Section \ref{sec:model definition} can be written as
\begin{align}
	v_{t+1,\a} (x) & = \min_{a\ge 0}\{KI_{\{a>0\}}+G_{t,\a}(x+a)\} - \bar{c}x ,\qquad t=0,1,\ldots,
	\label{eqn:vna}	
	\\
	\va (x) & =  \min_{a\ge 0}\{KI_{\{a>0\}}+\Ga(x+a)\} - \bar{c}x , \label{eqn:va}
\end{align}
and
\begin{align}
	G_{t,\a} (x) & = \bar{c}x + \E[h(x-D)] + \a\E[v_{t,\a}(x-D)] ,
	\label{eqn:Gna} \qquad t=0,1,\ldots, \\
	\Ga (x) & = \bar{c}x + \E[h(x-D)] + \a\E[\va(x-D)] ,\label{eqn:Ga}
\end{align}
and $v_{0,\a}(x) = 0$ for all $x\in\X,$ $\a \geq 0$  in equalities (\ref{eqn:vna}),
(\ref{eqn:Gna}), and $\a\in[0,1)$ in equalities (\ref{eqn:va}), (\ref{eqn:Ga});
e.g. see Feinberg and Lewis \cite{FEINBERG2015}.
The functions $G_{t,\a}$ and $\Ga$ are
lower semi-continuous for all $t=0,1,\ldots$ and $\a\in[0,1); $   Feinberg and Lewis
\cite[Corollary 6.4]{FEINBERG2015}.  Since all the costs are nonnegative,
equalities \eqref{eqn:Gna} and \eqref{eqn:Ga} imply that
\begin{align}
\lim_{x\to +\infty}G_{t,\a} (x)=\lim_{x\to +\infty}G_{\a} (x)=+\infty,
\qquad\qquad t=0,1,\ldots\ .
\label{eqn:limit Gta Ga}
\end{align}
Recall the definitions of $K$-convex functions and $(s,S)$ policies.
\begin{definition}\label{def:k-convex}
	A function $f:\X\to\R$ is called $K$-convex,  where $K\geq 0,$ if for each $x\leq y$
	and for each $\lambda\in(0,1),$
	\begin{displaymath}
		f((1 - \lambda)x+\lambda y) \leq (1 - \lambda)f(x)+\lambda f(y)+\lambda K	.
	\end{displaymath}
\end{definition}

Suppose $f$ is a lower semi-continuous $K$-convex function, such that
$f(x)\to \infty$ as $|x|\to\infty.$ Let
\begin{align}
	S &\in \underset{x\in\X}{\argmin} \{ f (x) \} \label{eqn:def S}, \\
	%\intertext{and}
	s &= \inf \{ x\leq S: f (x) \leq K + f(S) \}. \label{eqn:def s}
\end{align}

\begin{definition}\label{def:sS policy}
	Let $s_t$ and $S_t$ be real numbers such that $s_t\leq S_t,$ $t=0,1,\ldots\ .$
	A policy is called
	an $(s_t,S_t)$ policy at step $t$ if it orders up to the level $S_t,$ if $x_t<s_t,$
	and does not order, if $x_t\geq s_t.$ A Markov policy is called an $(s_t,S_t)$ policy
	if it is an $(s_t,S_t)$ policy at all steps $t=0,1,\ldots\ .$ A policy is called an
	$(s,S)$ policy if it is stationary and it is an $(s,S)$ policy at all steps
	$t=0,1,\ldots\ .$
\end{definition}

\begin{condition}[cp. Veinott and Wagner \cite{VEINOTT1965}]\label{cond:h(x)}
	There exist $z,$ $y\in\X$ such that $z<y$ and
	\begin{align}
		\frac{h(y) - h(z)}{y-z} < -\bar{c} .
		\label{eqn:condition 2.1}
	\end{align}
\end{condition}

It is well-known that for the problem considered in this paper and for relevant problems, this condition and its variations imply optimality $(s_t,S_t)$ policies for finite-horizon problems and $(s,S)$ policies for infinite horizon problems; see Scarf \cite{SCARF1960}, Iglehart \cite{IGLEHART1963} and Veinott and Wagner \cite{VEINOTT1965}. The following theorem presents the result from Feinberg and Lewis \cite{FEINBERG2015} for finite and infinite horizons and for arbitrary demand distributions; see also   Chen and Simchi-Levi \cite{CHEN2004a,CHEN2004b}, if price is fixed for the coordinating inventory control and pricing problems considered there.

%The following theorem presents the known results for the optimality of $(s,S)$ policies
%when condition \ref{cond:h(x)} holds.

\begin{theorem}[Feinberg and Lewis {\cite[Theorem 6.12]{FEINBERG2015}}]
	If Condition \ref{cond:h(x)} is satisfied, then the following statements hold:
	\begin{enumerate}[(i)]

         \item Let $\alpha\ge 0.$ For $t=0,1,\ldots$ consider real numbers
      $S_{t,\alpha}$ satisfying      \eqref{eqn:def S} and   $s_{t,\alpha}$ defined in
      \eqref{eqn:def s}  with $g(x)=G_{t,\alpha}(x),$ $x\in\X.$ Then for
      every $N=1,2,\ldots$  the  $(s_t,S_t)$ policy with $s_t=s_{N-t - 1,\alpha}$ and $S_t=S_{N-t - 1,\alpha},$
      $t=0,1,\ldots,N-1,$ is  optimal for the $N$-horizon problem.

		\item Let $\alpha\in [0,1).$ Consider real numbers $S_{\alpha}$ satisfying \eqref{eqn:def S}
          and $s_{\alpha}$ defined in   \eqref{eqn:def s} for $g(x):=G_\alpha(x),$
          $x\in\R.$ Then the
          $(s_\alpha,S_\alpha)$ policy is optimal for the infinite-horizon problem with the discount factor $\alpha.$ %The discount optimal $(s_\alpha,S_\alpha)$ defined in
      %Theorem~\ref{th:sSdisc} is optimal for each discount factor $\alpha\in [0,1).$
      %Furthermore, $s_\alpha=\lim_{t\to\infty}s_{t,\alpha}$ and any $(
%      s_\alpha,S^\prime_\alpha )$ polic y, where $S^\prime_\alpha$ is a limit point
%      of the sequence $\{S_{t,\alpha}\}_{t=0,1,\ldots},$ is optimal for the
%      discount factor $\alpha.$
 Furthermore, a sequence of pairs
          $\{(s_{t,\alpha},S_{t,\alpha})\}_{t=0,1,\ldots} $ considered in statement (i) is bounded, and, if
          $(s^*_\alpha,S^*_\alpha)$  is a limit  point of this sequence, then the
          $(s^*_\alpha,S^*_\alpha)$ policy is optimal for the infinite-horizon problem.
	\end{enumerate}
	\label{thm:sS policy cond holds}
\end{theorem}

If Condition \ref{cond:h(x)} does not hold, then finite-horizon optimal $(s_{t,\a},S_{t,\a})$ policies may not exist. It is shown Veinott and Wagner \cite{VEINOTT1965} for  discrete demand distributions and by Feinberg and Lewis
\cite{FEINBERG2015} for  arbitrary demand distributions that finite-horizon optimal $(s_{t,\a},S_{t,\a})$ policies exist if
certain non-zero terminal costs are assumed.

\begin{theorem}[Feinberg and Lewis {\cite[Theorem 6.10]{FEINBERG2015}}]
	There exists $\a^*\in[0,1)$ such that an $(s_{\a},S_{\a})$ policy is optimal
	for the infinite-horizon expected total discounted cost criterion with a discount
	factor $\a\in(\a^*,1),$ where the real numbers $S_{\a}$ satisfy \eqref{eqn:def S}   and $s_{\a}$ are defined by \eqref{eqn:def s}  with $f(x)=G_{\a}(x),$ $x\in\X.$
	Furthermore, a sequence of pairs ${(s_{t,\a},S_{t,\a})}_{t=0,1,\ldots},$  where the real numbers $S_{t,\a}$ satisfy \eqref{eqn:def S} and
		$s_{t,\a}$ are defined in \eqref{eqn:def s}
		with $f(x)=G_{t,\a}(x),$ $x\in\X,$ is bounded, and, for each its limit point $(\sa^*,\Sa^*),$ the $(\sa^*,\Sa^*)$ policy is optimal for the infinite-horizon problem with the discount factor $\a.$
	\label{thm:sS policy cond not hold:known}
\end{theorem}

\section{Structure of optimal polices}
\label{sec:structure general results}

This section describes the structure of finite-horizon and infinite-horizon optimal
policies. Unlike the previous section, it covers the situations when $(s,S)$ policies are
not optimal. Define
\begin{align}
	\a^* := 1 - \frac{k_h}{\bar{c}} ,
	\label{def a*}
\end{align}
where $k_h$ is introduced in \eqref{eqn:limit h(x)}. Since $0< k_h\leq +\infty,$
then $-\infty\leq \a^* < 1.$

\begin{lemma}\label{lm:k-h condition}
	Condition \ref{cond:h(x)} holds
	if and only if  $\alpha^*<0,$ which is equivalent to   $k_h > \bar{c}.$
\end{lemma}
\begin{proof} The inequalities $\alpha^*<0$ and $k_h>\bar{c}$ are equivalent because of $\bar{c}>0.$  Since $k_h>0,$ it is sufficient to prove
that Condition~\ref{cond:h(x)} does not hold if and only if
$k_h\in (0,c].$
Consider the function $H(y,z)$ defined in \eqref{eqn:H(y,z)} for $y,z\in\X$ such that $y\neq z.$

Let Condition \ref{cond:h(x)} do not hold.  Then we have $H(y,z) \geq -\bar{c}$ for
all $z<y.$ Since $H(0,x)$ is non-decreasing and bounded below by $-\bar{c}$ when $x<0,$
then $-\bar{c} \leq \lim_{x\to-\infty} H(0,x).$ Therefore,
$k_h=-\lim_{x\to-\infty}H(0,x)\in (0,\bar{c}]$ in view of \eqref{eqn:limit h(x)}.

Now, let us prove that $k_h\in (0,\bar{c}]$ implies that Condition \ref{cond:h(x)} does not hold.
Formula \eqref{eqn:limit h(x)} implies that
$\lim_{z\to -\infty} H(y,z)= -k_h$ for all $y\in\X.$ Since $H(y,z)$ is non-decreasing in
$z$ for $z<y,$ then $H(y,z)\geq -k_h \geq -\bar{c}$ for all $y,z\in\X$ satisfying $z<y.$
Therefore, Condition \ref{cond:h(x)} does not hold.
\end{proof}

Define the following function for all $t\in\N_0$ and $\a\geq 0,$
\begin{align}
	f_{t,\a}(x) := \bar{c}x + \sum_{i=0}^{t} \a^i \E[h(x-\bS_{i+1})] , \quad x\in\X.
	\label{eqn:fta(x)}
\end{align}
Observe that $f_{0,\a} (x)= \bar{c}x + \E[h(x-D)] = G_{0,\a} .$ Since $h(x)$ is a convex function,
then the function $f_{t,\a}(x)$ is convex for all $t\in\N_0$ and $\a\geq 0.$

Let $f_{t,\a}(-\infty):=\lim_{x\to -\infty}f_{t,\a}(x)$ and
\begin{align}
	N_{\a} := \inf \{ t\in \N_0: f_{t,\a}(-\infty) = +\infty \},
	\label{eqn:def Na}
\end{align}
where the infimum of an empty set is $+\infty.$ Since the function $h(x)$ is non-negative,
then the function $f_{t,\a}(x)$ is non-decreasing in $t$ for all $x\in\X$ and $\a\geq 0.$
Therefore, (i) $N_{\a}$ is non-increasing in $\a,$ that is, $N_{\a}\leq N_{\beta},$
if $\a>\beta;$ and (ii) in view of the definition of $N_{\a},$ for each $t\in\N_0$
\begin{align}
	f_{t,\a}(-\infty)  < +\infty , \qquad \text{if } t < N_{\a} , \qquad \text{and } \qquad
	f_{t,\a}(-\infty)	= +\infty, \qquad \text{if } t \geq N_{\a} .
	\label{eqn:limit of fta}
\end{align}
%Since $f_{0,\a} (x)= \bar{c}x + \E[h(x-D)] = G_{0,\a} ,$ if Condition \ref{cond:h(x)} holds, then $f_{0,\a} (-\infty) = +\infty$ which implies that $N_{\a} = 0.$
%Since $-\infty\leq \a^* <1,$ then Lemma \ref{lm:k-h condition} implies that Condition \ref{cond:h(x)} does not hold if and only if $\a^* \in [0,1).$
The following theorem provides the complete description of optimal finite-horizon  policies for all discount factors $\alpha.$
\begin{theorem}\label{thm:general results}    Let $\alpha>0.$
	Consider $\a^*$ defined in \eqref{def a*}. If $\a^*<0$ (that is, Condition \ref{cond:h(x)} is satisfied), then the statement of Theorem \ref{thm:sS policy cond holds}(i)
	holds. If $0\leq \a^* < 1,$ then the following statements hold for the finite-horizon problem with the discount factor $\alpha:$
	\begin{enumerate}[(i)]
		\item if $\a\in [0,\a^*],$ then the policy that never orders is optimal for
		every finite horizon $N=1,2,\ldots;$
\item if $\a > \a^*$, then $N_{\a}<+\infty$ and, for a finite horizon
		$N=1,2,\ldots,$  the following statements hold:
	\begin{enumerate}[(a)]
			\item if $N \leq N_{\a},$ then
			the policy that never orders  is optimal;
			
\item if $N> N_{\a},$ then a policy, that never orders at steps
			$t = N-N_{\a},N-N_{\a}+1,\ldots,N-1,$ and is an  $(s_t,S_t)$ policy with $s_t=s_{N-t-1,\a}$ and $S_t=S_{N-t-1,\a}$
			at steps $t = 0,1,\ldots,N-N_{\a}-1,$    is optimal, where the real numbers $S_{N-t-1,\a}$ satisfy
			 \eqref{eqn:def S} and  	$s_{N-t-1,\a}$ are defined in
			\eqref{eqn:def s}  with $f(x):=G_{N-t-1,\a}(x),$ $x\in\X$.
\end{enumerate}\end{enumerate}
\end{theorem}

\begin{remark}\label{rm:finite-horizon}
	For the $N$-horizon inventory control problem, according to Simchi-Levi et al. \cite[Theorem 8.3.4, p. 126]{SCB}, $(s_t,S_t)$ policies, $t=0,1,\ldots,N,$ are optimal. However, in the model considered there, the inventory left at time $N$ has a salvage value $\bar{c}$ per unit. Furthermore, in that formulation, $\alpha = 1$ implies that the value of $\bar{c}$ does not affect the decisions. Let us take $\bar{c}>0$ small enough to have $\alpha^* <0.$ Then theorem \ref{thm:general results} also implies the optimality of $(s_t,S_t)$ policies, $t=0,1,\ldots,N.$
\end{remark}

The conclusions of Theorem~\ref{thm:general results}  are presented in Table~\ref{Table1}. If the discount factor $\a\in [0,1),$ the conclusions of Theorem~\ref{thm:general results}  are presented in Figure~\ref{fig:optPol_finite},
in addition, if the discount factor $\a\geq 1,$ then the case presented in the last column of the Table~\ref{Table1}, that is if $1>\alpha^* \geq \alpha,$ is impossible, and the conclusions for $\alpha\ge 1$ are presented in Figure~\ref{fig:optPol_finite1}.

\begin{table}[ht]
\centering
\caption{The structure of optimal policies for a discounted $N$-horizon problem with $N<+\infty$ and $\alpha\ge 0.$}\label{Table1}

	\begin{tabular}{|c|l|l|l|}
    	\hline
    $\alpha$ & \multicolumn{1}{|c|} {$\alpha^* < 0$} &  \multicolumn{1}{|c|}{$0\leq \alpha^*<\alpha$} & \multicolumn{1}{|c|}  {$1>\alpha^* \geq \alpha$} \\
	\hline
& There is  & For the natural number $N_{\a}$ defined in \eqref{eqn:def Na},  & The policy\\
& an optimal & $\ \ \ $ if $N > N_{\alpha},$ then a policy, that never orders at & that never \\
& $(s_{t},S_{t})$   & steps $t = N-N_{\alpha},\ldots,N-1$ and is an $(s_{t},S_{t})$ & orders is  \\
& policy. & policy at steps $t = 0,\ldots,N-N_{\alpha}-1,$ is optimal; &  optimal. \\
&	&  $\ \ \ $ if $N \leq N_{\alpha},$ then a policy that never orders is  & \\
&	&  optimal. & \\
		\hline
	\end{tabular}
	\label{tab:optimal policies T horizon}

\end{table}

\begin{figure}[ht]
  \centering
  \caption{The structure of optimal policies for a discounted $N$-horizon problem with $N < +\infty$ and $\alpha \in [0,1)$.}
  \label{fig:optPol_finite}
  \includegraphics[scale=0.3]{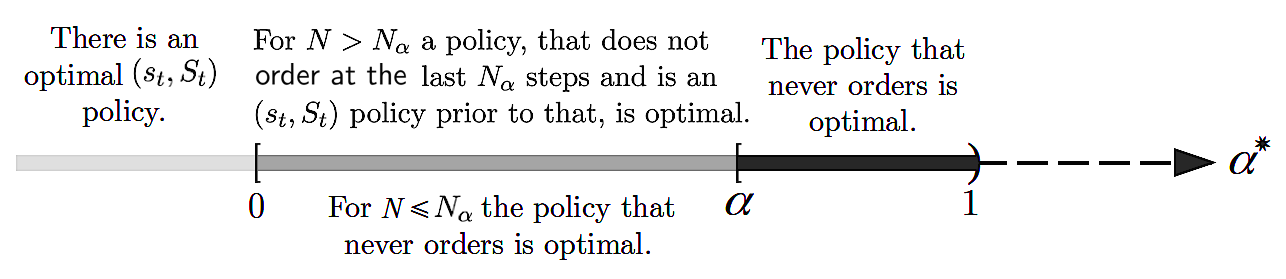}
\end{figure}

\begin{figure}[ht]
  \centering
  \caption{The structure of optimal policies for a discounted $N$-horizon problem with $N < +\infty$ and $\alpha \geq 1$.}
  \label{fig:optPol_finite1}
  \includegraphics[scale=0.3]{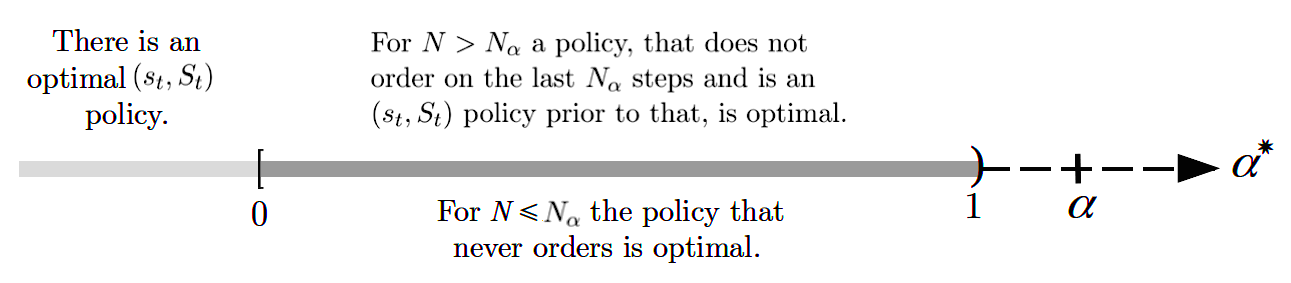}
\end{figure}

The following theorem describes optimal
infinite-horizon policies for all discount factors $\alpha\in [0,1).$

\begin{theorem}\label{thm:general results_ig} Let $\alpha\in [0,1).$
	Consider $\a^*$ defined in \eqref{def a*}. The following statements hold for the infinite-horizon problem with the discount factor $\alpha:$
\begin{enumerate}[(i)]
\item if $\a^*<\a,$ then an $(s_\alpha, S_\alpha)$ policy is optimal, where the real numbers $S_{\a}$ and 	$s_{a}$ are defined in \eqref{eqn:def S} and \eqref{eqn:def s} respectively with $f(x):=G_{\a}(x),$ $x\in\X.$
Furthermore, a sequence of pairs ${(s_{t,\a},S_{t,\a})}_{t=N_{\a},N_{\a}+1,\ldots}$ considered in Theorem \ref{thm:general results} (ii,b) is bounded, and, for if
$(\sa^*,\Sa^*)$ is a limit point of the sequence, then the $(\sa^*,\Sa^*)$ policy is optimal for the infinite-horizon problem with the discount factor $\a;$
\item if $\alpha^*\ge\alpha,$ then the policy that never orders is optimal.
\end{enumerate}
\end{theorem}

The conclusions of Theorem~\ref{thm:general results_ig}  are presented in Table~\ref{Table2} and Figure~\ref{fig:optPol_infinite}.  To prove Theorems~\ref{thm:general results} and \ref{thm:general results_ig}, we first establish several auxiliary statements.
%To prove Theorems \ref{thm:general results} and \ref{thm:general results_ig} we need several auxiliary lemmas.

\begin{table}[ht]
\centering
\caption{The structure of optimal policies for a discounted infinite-horizon problem with $\alpha\in [0,1)$.}\label{Table2}
	\begin{tabular}{|c|l|l|}
    	\hline
    $\a$ & \multicolumn{1}{|c|}{$\a^*<\a$} & \multicolumn{1}{|c|}{$\alpha \leq \alpha^*$} \\
	\hline
	& There is an optimal  & The policy that never \\
	& $(s,S)$ policy. &  orders is optimal. \\
	\hline
	\end{tabular}
	\label{tab:optimal policies infinite horizon}
	
\end{table}

\begin{figure}[ht]
  \centering
  \caption{The structure of optimal policies for a discounted infinite-horizon problem with $\alpha \in [0, 1)$.} % Here  we use the notation $(s_{\alpha},S_{\alpha})$ instead %of  $(s,S)$ to indicate that the thresholds depend on $\alpha.$}
  \label{fig:optPol_infinite}
  \includegraphics[scale=0.3]{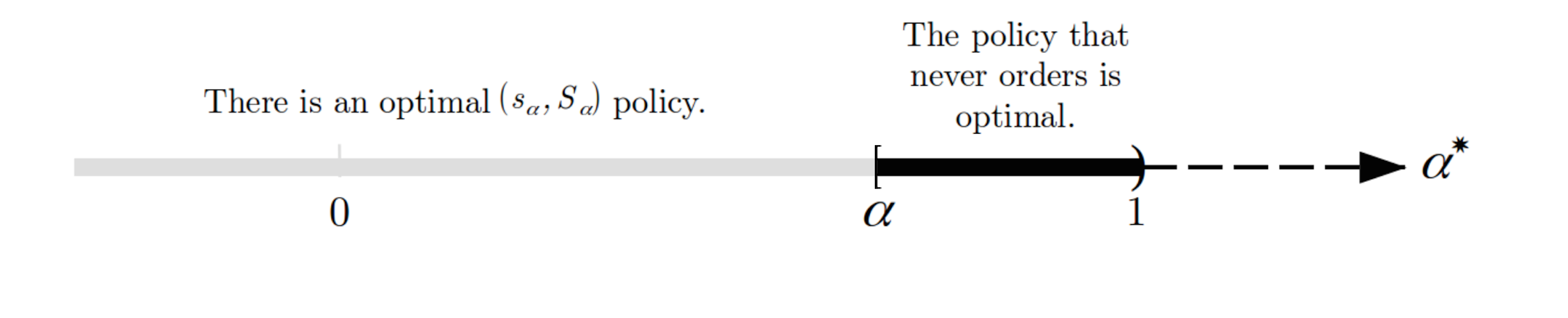}
\end{figure}

\begin{lemma}\label{lm:1}
	If Condition \ref{cond:h(x)} does not hold, then:
	\begin{enumerate}[(i)]
		\item $\E[h(x-\bS_t)] \leq \bar{c}t\E[D]-\bar{c}x$ for all $x\leq 0$ and $t=0,1,\ldots;$	
		\item for each $\e\in (0,k_h)$ there exists a number $M_{\e} < 0$ such that
		for all $z<y\leq M_{\e}$ and $t = 0,1,\ldots,$
		\begin{equation}
			-k_h \leq \frac{\E[h(y-\bS_t) - h(z-\bS_t)]}{y-z} < -k_h + \e <0 .
			\label{eqn:h(x):2}
		\end{equation}
	\end{enumerate}
\end{lemma}
	
% proof
%\paragraph{Proof of Lemma \ref{lm:1}.}
\begin{proof}
(i) Consider the function $H(y,z)$ defined in \eqref{eqn:H(y,z)} for all $y,z\in\X$ satisfying
$y\neq z.$ According
to Lemma \ref{lm:k-h condition}, since Condition \ref{cond:h(x)} does not hold, then
$k_h\in(0,\bar{c}].$ Since $H(y,z)$ is non-decreasing in $y$ on $\X\setminus \{z\}$
for all $z\in\X$ and non-decreasing in $z$ on $\X\setminus \{y\}$ for all $y\in\X,$ then
$H(0,x)=\frac{h(x)}{x}\geq -\bar{c}$ for all $x<0,$ which is equivalent to $h(x)\leq -\bar{c}x$
for all $x\leq 0.$ Let $x\le 0.$  Then $x-\bS_t\leq 0$ almost surely (\textit{a.s.})
for all $t=0,1,\ldots\ .$ Thus
$\E[h(x-\bS_t)] \leq \E[-\bar{c} (x-\bS_t)] = \bar{c}t\E[D] - \bar{c}x$ for all $x\leq 0$ and $t=0,1,\ldots\ .$

(ii) Since $\lim_{x\to-\infty} H(0,x)=-k_h$ and $H(0,x)$ is non-decreasing when $x<0,$
then for each $\e\in (0,k_h)$ there exists $M_{\e}<0$ such that
$-k_h \leq H(0,M_{\e}) < -k_h + \e<0.$ Therefore,
$H(y,z)\leq H(0,z)\leq H(0,M_{\e})< -k_h + \e$ for all $z<y\leq M_{\e},$ where the first
two inequalities follow from the monotonicity properties of $H(y,z)$ stated in the first
paragraph of the proof. As follows from \eqref{eqn:limit h(x)},
$\lim_{z\to -\infty} H(y,z)= -k_h.$ Since the function $H(y,z)$ is non-decreasing
in $z$ when $z<y,$  for all $z<y\leq M_{\e},$
\begin{equation}
	-k_h \leq H(y,z) = \frac{h(y) - h(z)}{y-z} < -k_h + \e . \label{eqn:h(x):1}
\end{equation}

Since $\bS_t \geq 0$ \textit{a.s.} for all $t=0,1,\ldots,$ then \eqref{eqn:h(x):1} implies
that  $-k_h \leq H(y-\bS_t,z-\bS_t) < -k_h + \e $ \textit{a.s.}
for all $z<y\leq M_{\e},$  which yields
$-k_h \leq \E[H(y-\bS_t,z-\bS_t)] < -k_h + \e .$  The last inequalities are equivalent to \eqref{eqn:h(x):2}.
\end{proof}

\begin{lemma}\label{lm:2}
	If the function $G_{t,\a}(x)$ is convex in $x$ and $\lim_{x\to -\infty}G_{t,\a}(x)<+\infty$
	for some $t=0,1,\ldots,$ then  for the epoch $t$ the minimum in the optimality equation \eqref{eqn:vna} is achieved for all $x\in\X$ at the action $a=0,$
	and the functions $v_{t+1,\a} (x)$ and $G_{t+1,\a} (x)$ are convex in $x.$
\end{lemma}

% proof
%\paragraph{Proof of Lemma \ref{lm:2}.}
\begin{proof}
Since $G_{t,\a}(x)$ is a convex function and $\lim_{x\to -\infty}G_{t,\a}(x)<+\infty,$ then
the function $G_{t,\a}(x)$ is non-decreasing on $\X.$ Therefore,
$G_{t,\a}(x) \leq K + G_{t,\a}(x+a)$ for all $x\in\X$ and $a\geq 0.$ In view of
\eqref{eqn:vna}, the action $a=0$ is optimal at the epoch $t.$ Therefore,
$
	v_{t+1,\a} (x)=G_{t,\a}(x)-\bar{c}x .
$	%\label{eqn:not order}
%\end{align}
This formula and convexity of the function $G_{t,\a}(x)$ imply that
$v_{t+1,\a} (x)$ is a convex function, which, in view of \eqref{eqn:Gna},
implies that the function $G_{t+1,\a} (x)$ is convex.
\end{proof}

\begin{lemma}\label{lm:3}
	Let $\a >0$ and there exists $t_0=0,1,\ldots$ such that
	the function $G_{t_0,\a}(x)$ is $K$-convex and $\lim_{x\to -\infty} G_{t_0,\a}(x)=+\infty.$
	Then  the functions $G_{t,\a}(x),$ $t=t_0,t_0+1,\ldots,$  are $K$-convex and
 $G_{t,\a}(x)\to +\infty$ as $|x|\to +\infty.$ %for all $t=t_0,t_0+1,\ldots,$ and (ii) the function
%	$\Ga (x)$ is $K$-convex and  $\Ga (x)\to +\infty$ as $|x|\to +\infty.$
\end{lemma}

\begin{proof}
%Since all the costs are non-negative, then $v_{t,\a}(x)\geq 0$ for all $t=0,1,\ldots$ and
%$x\in\X.$ Equality \eqref{eqn:Gna} implies that $G_{t,\a}(x)\to +\infty$ as $x\to +\infty$ for all $t=1,2,\ldots\ .$
Since all the costs are nonnegative,
$v_{t,\a}(x)\le v_{t+1,\a}(x)$
and therefore
$G_{t,\a}(x)\le G_{t+1,\a}(x)$
for all $x\in\X$ and for all $t=0,1,\ldots\ .$  This implies
$\lim_{|x|\to +\infty} G_{t,\a}(x)=+\infty $ for $t=t_0,t_0+1,\ldots\ .$
Assume that the function $G_{t,\a}(x)$ is $K$-convex for some $t\geq t_0.$ In view of
Heyman and Sobel \cite[Lemma 7-2]{HEYMAN1984}, since
$\lim_{|x|\to +\infty} G_{t,\a}(x)=+\infty$ and $G_{t,\a}(x)$ is $K$-convex, then
the function $v_{t+1,\a}(x)$ is $K$-convex. Therefore,
\eqref{eqn:Gna} implies that the function $G_{t +1,\a}(x)$ is $K$-convex. Since
$G_{t_0,\a}(x)$ is $K$-convex, then the induction arguments imply that the functions
$G_{t,\a}(x)$ are $K$-convex for all $t=t_0,t_0+1,\ldots\ .$
\end{proof}
\begin{lemma}\label{lm:32}
	Let $\a\in [0,1)$ and there exists $t_0=0,1,\ldots$ such that
	the function $G_{t_0,\a}(x)$ is $K$-convex and $\lim_{x\to -\infty} G_{t_0,\a}(x)=+\infty.$
	Then  the function $G_{\a}(x)$ is $K$-convex and
 $G_{\a}(x)\to +\infty$ as $|x|\to +\infty.$ %for all $t=t_0,t_0+1,\ldots,$ and (ii) the function
%	$\Ga (x)$ is $K$-convex and  $\Ga (x)\to +\infty$ as $|x|\to +\infty.$
\end{lemma}
\begin{proof}
In view of Lemma~\ref{lm:3}, the functions $G_{t,\a}(x),$ $t=t_0,t_0+1,\ldots,$  are $K$-convex and
 $G_{t,\a}(x)\to +\infty$ as $|x|\to +\infty.$ According to Feinberg et al. \cite[Theorem 2]{FEINBERG2012}, $v_{t,\a}(x)\uparrow v_{\a}(x)$ as $t\to +\infty$ and therefore $G_{t,\a}(x)\uparrow G_{\a}(x)$ as $t\to +\infty.$ The $K$-convexity of the functions $G_{t,\a}$ stated in Lemma~\ref{lm:3} implies the $K$-convexity of the function $G_\a.$
\end{proof}

\begin{lemma}\label{lm:G=f} For $\alpha\ge 0$
	consider the function $f_{t,\a}(x)$ and number $N_{\a}$ defined in
	\eqref{eqn:fta(x)} and \eqref{eqn:def Na}.   The following statements hold:
	\begin{enumerate}[(i)]
		\item if $N_{\a}<+\infty,$ then $G_{t,\a}(x)=f_{t,\a}(x)$ for all $t=0,1,\ldots,N_{\a};$
		\item if $N_{\a}=+\infty,$ then $G_{t,\a}(x)=f_{t,\a}(x)$ for all $t=0,1,\ldots\ .$
	\end{enumerate}
\end{lemma}

\begin{proof}
Let us prove this lemma by induction. As stated after formula \eqref{eqn:fta(x)},
$G_{0,\a}(x)=f_{0,\a}(x),$ $x\in\X.$ Now assume that $G_{k,\a}(x) = f_{k,\a}(x),$ $x\in\X,$
for some $k\in\N_0$ satisfying $k <N_{\a}.$ Then
\begin{align}
\begin{split}
G_{k+1,\a}(x)
& = \bar{c}x + \E[h(x-D)] + \a\E[ G_{k,\a} (x-D) - \bar{c}(x-D) ] \\
& = \bar{c}x + \E[h(x-\bS_1)] + \a\E[ f_{k,\a} (x-D) - \bar{c}(x-D) ] \\
& = \bar{c}x + \sum_{i=0}^{k+1} \a^i \E[h(x-\bS_{i+1})] = f_{k+1,\a}(x),
\end{split}
\label{eqn:induction}
\end{align}
where the first equality follows from Lemma \ref{lm:2} and equations \eqref{eqn:Gna}, \eqref{eqn:limit of fta}, the second one  follows from the induction assumption, and the last two equalities  follow from \eqref{eqn:fta(x)}.
Hence the induction arguments imply the conclusions in statements (i) and (ii).
\end{proof}

\begin{lemma}\label{pro:1}
	Consider $\a^*=1-\frac{k_h}{c}$  defined in \eqref{def a*}.
	Let Condition \ref{cond:h(x)} do not hold. Then the following statements hold:	
	\begin{enumerate}[(i)]
	\item if $\a>\a^*,$ then $1\leq N_{\a}<+\infty;$
	\item if $\a\in [0,\a^*],$ then $N_{\a}=+\infty$ and, in addition, the function $G_{\a}(x)$ is
	convex and $\lim_{x\to -\infty} G_{\a}(x)<+\infty.$
	\end{enumerate}		
\end{lemma}

\begin{proof}
According to Lemma \ref{lm:k-h condition}, since Condition \ref{cond:h(x)} does not hold,
then $\a^* \geq 0.$

(i) If $\a>\a^*,$ then there exists $\delta > 0$ such that $\a > \a^* + \delta$ and
$\bar{c}\delta < k_h.$ Let $\e_k = \bar{c}\delta\in (0,k_h).$ Then
$0\leq \a^* < \a^* + \delta = 1+\frac{\e_k-k_h}{c}< \a  .$
According to Lemma \ref{lm:1}(ii), for $\e_k\in(0,k_h),$ there exists $M_{\e_k}<0$ such that
\eqref{eqn:h(x):2} holds for all $z<y\leq M_{\e_k}$ and $t = 0,1,\ldots\ .$ Therefore, for
all $ z<y\leq M_{\e_k}$ and $t = 0,1,\ldots,$
\begin{align}
	\sum_{i=0}^{t}\a^i \frac{\E[h(y-\bS_{i+1}) - h(z-\bS_{i+1})]}{y-z} <
	(-k_h + \e_k ) \sum_{i=0}^{t}\a^i < 0.
	\label{eqn:limit of f}
\end{align}
If $\a\in(\a^*,1),$ then $(-k_h + \e_k )\sum_{i=0}^{+\infty}\a^i = \frac{-k_h + \e_k}{1-\a} <
\frac{-k_h + \e_k}{1-(\a^*+\d)}= - \bar{c}.$ If $\a \geq 1,$ then
$(-k_h + \e_k )\sum_{i=0}^{+\infty}\a^i = -\infty < -\bar{c}.$ Therefore, for all $\a>\a^*$
there exists a natural number $M$ such that
\begin{align}
	(-k_h + \e_k )\sum_{i=0}^{M}\a^i<-\bar{c} .
	\label{eqn:T 1}
\end{align}
Thus, \eqref{eqn:limit of f} and \eqref{eqn:T 1} imply that there exist $y,$ $z\in\X$
satisfying $z<y$ such that
$\sum_{i=0}^{M}\a^i \frac{\E[h(y-\bS_{i+1}) - h(z-\bS_{i+1})]}{y-z} < -\bar{c} ,$
which is equivalent to
\begin{align}
	f_{M,\a}(y) - f_{M,\a}(z) &< 0 .
	\label{eqn:fy < fz}
\end{align}
Since the function $f_{M,\a}(x)$ is convex, then \eqref{eqn:fy < fz} implies that
$f_{M,\a}(-\infty)=+\infty.$ Therefore, $N_{\a}\leq M <+\infty.$ Since
Condition \ref{cond:h(x)} does not hold, then $f_{0,\a}(-\infty) < +\infty.$
Therefore, $N_{\a}\geq 1.$

(ii) Consider $\a\in[0,\a^*].$ According to Lemma \ref{lm:1}(ii), for $\e>0,$
there exists $M_{\e}<0$ such that for all $z<y\leq M_{\e}$ and $t = 0,1,\ldots,$
\begin{equation}
	\begin{split}
	&  \sum_{i=0}^{t}\a^i \frac{\E[h(y-\bS_{i+1}) - h(z-\bS_{i+1})]}{y-z}
	 \geq
	\sum_{i=0}^{+\infty}\a^i \frac{\E[h(y-\bS_{i+1}) - h(z-\bS_{i+1})]}{y-z} \\
	 \geq & -k_h \sum_{i=0}^{+\infty}\a^i  = \frac{-k_h}{1-\a} \geq \frac{-k_h}{1-\a^*} =-\bar{c} ,	
	\end{split}
	\label{eqn:small df lower bd on slope of f}
\end{equation}
where the first two inequalities follow from \eqref{eqn:h(x):2}, the first equality and
the last inequality are straightforward, and the last equality follows from the definition
of $\a^*.$ In view of \eqref{eqn:fta(x)}, \eqref{eqn:small df lower bd on slope of f} equivalent to
$f_{t,\a}(y) \geq f_{t,\a}(z)$ for all $t\in\N_0$ and $z<y\leq M_{\e}.$
Therefore, $f_{t,\a}(-\infty)<+\infty$ for all $t=0,1,\ldots,$ which implies that
$N_{\a}=+\infty.$

According to Feinberg et al. \cite[Theorem 2]{FEINBERG2012}, $v_{t,\a}(x)\uparrow v_{\a}(x)$
as $t\to +\infty$ and therefore $G_{t,\a}(x)\uparrow G_{\a}(x)$ as $t\to +\infty.$ Therefore,
in view of Lemma \ref{lm:G=f}(ii),
\begin{align}
	\Ga(x)=\bar{c}x + \sum_{i=0}^{+\infty} \a^i \E[h(x-\bS_{i+1})] ,
	\label{eqn:ga=f infty}
\end{align}
which implies that the function $\Ga(x)$ is convex. Observe that
$\Ga(0) \leq \sum_{i=0}^{+\infty}\a^i \bar{c}(i+1)\E[D]= \bar{c}\E[D]\frac{\a(2-\a)}{{(1-\a)}^2}< +\infty,$
where the first inequality follows from \eqref{eqn:ga=f infty} and Lemma \ref{lm:1}(i).
In view of \eqref{eqn:small df lower bd on slope of f}, it is equivalent to
$\Ga(y) \geq \Ga(z)$ for all $z<y\leq M_{\e}.$ Therefore, since the function $G_\alpha(x)$ is convex,
$\lim_{x\to -\infty} G_{\a}(x)<+\infty.$
\end{proof}

\begin{proof}[Proof of Theorem {\ref{thm:general results}}]
Let  $N=1,2,\ldots$ be the horizon length.  Consider the parameter $\a^*$ defined in \eqref{def a*}. If $\a^* <0,$ then Lemma \ref{lm:k-h condition}
implies that Condition \ref{cond:h(x)} holds. Therefore, the results follow from Theorem
\ref{thm:sS policy cond holds}(i). On the other hand, if $0\leq \a^* <1,$ then Lemma
\ref{lm:k-h condition} implies that Condition \ref{cond:h(x)} does not hold.

(i) Suppose $\a\in[0,\a^*].$ Lemma \ref{lm:G=f}(ii) and Lemma
\ref{pro:1}(ii) imply that $G_{t,\a}(x),$ $t=0,1,\ldots,N-1,$ are convex functions and
$\lim_{x\to -\infty}G_{t,\a}(x) < +\infty.$  Therefore,
in view of Lemma \ref{lm:2}, a policy that never orders at steps $t=0,1,\ldots,N-1$ is optimal.

(ii) Suppose $\a > \a^*.$ Lemmas \ref{lm:2}, \ref{lm:G=f}(i), and \ref{pro:1}(i) imply that
(a) if $N\leq N_{\a},$ then a policy that never orders at steps $t=0,1,\ldots,N-1$ is optimal,
and (b) if $N> N_{\a},$ then the action $a=0$ is always optimal at steps
$t=N-N_{\a},\ldots,N-1$ and furthermore $N_{\a}<+\infty$ and
$G_{N_{\a},\a}(x)=f_{N_{\a},\a}(x).$ In view of Lemma \ref{lm:3},
the functions $G_{t,\a}(x),$ $t=N_{\a},N_{\a}+1,\ldots,$
are $K$-convex and $\lim_{|x|\to +\infty} G_{t,\a}(x)=+\infty.$ These properties of the functions $G_{t,\alpha}(x)$ imply
the optimality of $(s_t,S_t)$ policies at steps  $t=N-N_{\a},\ldots,N-1$   described in
statement (ii-b); see, e.g, the paragraph following the proof of Proposition 6.7 in Feinberg and Lewis \cite{FEINBERG2015}.
\end{proof}

\begin{proof}[Proof of Theorem {\ref{thm:general results_ig}}]
Consider an infinite-horizon problem and the parameter $\a^*$ defined in \eqref{def a*}. If $\a^* <0,$ then Lemma \ref{lm:k-h condition}
implies that Condition \ref{cond:h(x)} holds. Therefore, statement (i) follows from Theorem \ref{thm:sS policy cond holds}(ii).
On the other hand, if $0\leq \a^* <1,$ then Lemma
\ref{lm:k-h condition} implies that Condition \ref{cond:h(x)} does not hold.

(i) Suppose $\a > \a^*.$ Lemma \ref{lm:G=f}(i)
and Lemma \ref{pro:1}(i) imply that $N_{\a}<+\infty$ and $G_{N_{\a},\a}(x)=f_{N_{\a},\a}(x).$ Therefore, according to Lemma \ref{lm:32},
the function $G_{\a}(x)$ is $K$-convex and $\lim_{|x|\to +\infty} \Ga (x) = +\infty,$ and this implies
statement (i). % follows from Feinberg and Lewis \cite[Lemma 3.5, Proposition 6.7, Theorem 6.11(ii)]{FEINBERG2015}.

(ii) Suppose $\a\in[0,\a^*].$  According to Lemma \ref{pro:1}(ii), the function
$G_{\a}(x)$ is convex and $\lim_{x\to -\infty}G_{\a}(x) < +\infty.$ Therefore, this function
 is nondecreasing. Therefore, $G_{\a}(x) \leq K + G_{\a}(x+a)$
for all $x\in\X$ and $a\geq 0,$ and this implies that a policy that never orders is optimal.
\end{proof}

\section{Continuity of the value functions}
\label{sec:continuity}

In this section we show that the value functions $v_{N,\alpha}(x),$ $N=1,2,\ldots,$ and
$v_\alpha (x)$ are continuous in $x\in\X.$ As explained in Feinberg and Lewis
\cite[Corollary 6.1]{FEINBERG2015}, the general results on MDPs imply that
these functions are inf-compact and therefore they are lower semi-continuous.
As discussed above, these functions are $K$-convex.  However,
Example~\ref{ex:continuity and k-convexity} illustrates that a $K$-convex
function may not be continuous.  Thus, the continuity of the value functions
$v_{N,\alpha}(x),$ $N=1,2,\ldots,$ and $v_\alpha (x)$ follows neither from the known general
properties of value functions for infinite-state MDPs nor from these properties combined
with the $K$-convexity of these functions.

We recalled that a function $f:\U\to \R\cup \{+\infty\}$ for a metric space $\U$ is
called lower semi-continuous, if the level set
\begin{align}
	\mathfrak{D}_{f} (\lambda) := \{u\in\U :f(u)\leq \lambda \} ,
	\label{eqn:level set}
\end{align}
is closed for every $\lambda\in\R.$ A function $f:\U\to \R\cup \{+\infty\}$ is called
inf-compact, if the level set $\mathfrak{D}_{f} (\lambda)$ defined in \eqref{eqn:level set}
is compact for every $\lambda\in\R.$ Of course, each inf-compact function is lower
semi-continuous.  As proved in Feinberg and Lewis \cite{FEINBERG2015} (see also see
Feinberg et al. \cite[Theorem 2]{FEINBERG2012}), for an MDP with a standard Borel state
space, if the one-step costs function $c:\X\times \A\to \Rp \cup\{+\infty\}$ is inf-compact and
the transition probabilities $p(\cdot|x,a)$ are weakly continuous in $(x,a)\in \X\times \A,$
then the value functions  $v_{N,\alpha}(x),$ $N=1,2,\ldots,$ and $v_\alpha (x)$ are
inf-compact in $x\in \X,$ where $\X$ and $\A$ are the state and action sets of the MDP.
In particular, for our problem the function $c$ defined in
\eqref{inventory-control:cost function} is inf-compact, and the transition probabilities are
weakly continuous; Feinberg and Lewis \cite{FEINBERG2015}.  This implies that these functions
are lower semi-continuous; see also Feinberg et al. \cite[Theorem 2]{FEINBERG2012} for a
more general statement on lower semi-continuity of the value functions. It is easy to
provide an example, when the value function is not continuous for an MDP with inf-compact one-step
costs and weakly continuous transition probabilities; see Feinberg et al. \cite[Example 4.4]{FKZ2013}.

%\begin{example}
%	Consider $\X=\R$ and $A(x)=\A=\{0\}$ for all $x\in\X.$ The one-step costs are
%	\begin{align*}
%		c(x,0) =
%		\begin{cases}
%			0	& \text{if } x = 0 , \\
%			|x| + 1	& \text{if } x\neq 0 .
%		\end{cases}
%	\end{align*}
%	The transitions are $p(0|x,0)=1$ for all $x\in\X,$ which are weakly continuous.
%	Since the level set
%	$\mathfrak{D}_c (\lambda) = [-\max\{\lambda-1,0\},\max\{\lambda-1,0\}]\times\{0\}$ if
%	$\lambda\geq 0,$ and is an empty set if $\lambda < 0,$ then the one-step costs are
%	inf-compact. Since $\{ 0 \}$ is an absorbing state with zero costs,
%	then $v_{N,\a}(x) = \va(x) = c(x,0)$
%	for all $N=1,2,\ldots$ and $x\in\X.$ Therefore, the value functions $v_{N,\a}(x),$
%	$N=1,2,\ldots,$ and $\va(x)$ are not continuous at $x=0.$
%	\label{ex:continuity and MDP}
%\end{example}

The following example demonstrates that $K$-convex function may not be continuous. So, the $K$-convexity of
the function $v_{t,\a}$ and $v_{\a}$ does not imply their continuity.

\begin{example}\label{ex:continuity and k-convexity}
	For fixed $K>0$ and $d\in [0,K],$  the following discontinuous function
	\begin{align*}
		f (x) =
		\begin{cases}
			- x  + K			& \text{if }  x <0 ,  \\
			d			& \text{if }  x = 0, \\
			x 	& \text{if }  x > 0 ,
		\end{cases}
	\end{align*}	
is $K$-convex.  To verify $K$-convexity, observe that this function is convex on $(-\infty,0)$ and  $[0,+\infty).$  Let $x<0,$  $y\ge 0,$ and $\lambda\in (0,1).$ If $(1-\lambda)x+\lambda y\ne 0,$ then $f((1-\lambda)x+\lambda y)\le |(1-\lambda)x+\lambda y|+K \le (1-\lambda)f(x)+\lambda f(y)+\lambda K.$ If $(1-\lambda)x+\lambda y= 0,$
then $f(0)\le K < (1-\lambda)(-x) +\lambda y + K=(1-\lambda)f(x)+\lambda f(y)+\lambda K.$
\end{example}

The following theorem describes the continuity of value functions for finite-horizon inventory control problems considered in this paper. The continuity of finite-horizon value functions is proved by induction. From Theorem \ref{thm:general results}, we know that either $(s_t,S_t)$ policy or a policy that does not order is optimal at epoch $t.$ We prove that under these two cases the value function $v_{t+1}$ is continuous if $v_t$ is a continuous function.

\begin{theorem}\label{thm:cont finite-horizon}
For a finite horizon inventory control problem, the functions
$v_{t,\a}(x)$ and $G_{t,\a} (x),$ $t=0,1,\ldots,$ are continuous on $\X$ for all
$\a\geq 0.$
\end{theorem}

\begin{proof}
We prove by induction that the functions $v_{t,\a}(x)$ and $G_{t,\a} (x),$ $t=0,1,\ldots,$
are continuous. Let $t=0.$ Then $v_{0,\a}(x)=0$ and $G_{0,\a}(x)=\bar{c}x + \E[h(x-D)],$ $x\in\X.$
Therefore, for all $\a\geq 0,$ the functions $v_{0,\a}(x)$ and $G_{0,\a}(x)$ are convex on
$\X$ and hence they are continuous.

Now assume that $v_{t,\a}(x)$ and $G_{t,\a}(x)$ are continuous functions for some $t\geq 0.$
According to Theorem \ref{thm:general results} and Lemma \ref{pro:1}, one of the
following cases takes place: (i) $G_{t,\a}(x)$ is a convex function,
$\lim_{x\to -\infty}G_{t,\a}(x)<+\infty,$ and the action $a=0$ is optimal at all states when
$t$ periods are left or (ii) $G_{t,\a}(x)$ is a $K$-convex function,
$\lim_{|x|\to +\infty}G_{t,\a}(x)=+\infty,$ and the $(s_{t,\a},S_{t,\a})$ policy is optimal
when $t$ periods are left, where $s_{t,\a}$ and $S_{t,\a}$
are defined in \eqref{eqn:def s} and \eqref{eqn:def S} with $f(x):=G_{t,\a}(x),$ $x\in\X.$

Case (i). In view of \eqref{eqn:vna}, since an action that never orders is optimal, then
$v_{t+1,\a}(x) = G_{t,\a}(x) -\bar{c}x.$ Therefore, convexity of the function $G_{t,\a}(x)$ implies
that $v_{t+1,\a}(x)$ is a convex function. In view of \eqref{eqn:Gna}, since $h(x)$ and
$v_{t+1,\a}(x)$ are convex functions on $\X,$ then $G_{t+1,\a}(x)$ is also a convex function.
Since the functions $v_{t+1,\a}(x)$ and $G_{t+1,\a}(x)$ are convex on $\X.$ Thus they are
continuous.

Case (ii). Since there exists an optimal $(s_{t,\a},S_{t,\a})$ policy, then in view of
\eqref{eqn:vna}, the function $v_{t+1,\a}(x)$ can be written as
\begin{align}
	v_{t+1,\a} (x) =
	\begin{cases}
		G_{t,\a} (x) - \bar{c}x  				& \text{if } x \geq s_{t,\a} ,  \\
		K + G_{t,\a} (S_{t,\a}) - \bar{c}x  	& \text{if } x < s_{t,\a} .  \\
	\end{cases}
	\label{eqn:vna sS}
\end{align}
In view of the definition of $s_{t,\a}$ in \eqref{eqn:def s} with $f(x):=G_{t,\a}(x),$
$x\in\X,$ since the function $G_{t,\a}(x)$ is continuous on $\X,$ then
$G_{t,\a}(s_{t,\a})=K+G_{t,\a} (S_{t,\a}).$
Therefore, \eqref{eqn:vna sS} implies that the function $v_{t+1,\a} (x)$ is continuous.

Let us prove that the function $G_{t+1,\a}$ is continuous.  It is sufficient to prove that this function is continuous on each interval $(-\infty,b)$, where $b\in \R.$  Let us fix an arbitrary real number $b.$

Let us consider the continuous function $g_{t+1,\a}(x)= v_{t+1,\a}(x) + \bar{c}x.$ This function is bounded on $(-\infty,b)$ because, in view of \eqref{eqn:vna sS},
$g_{t+1,\a}(x)=K + G_{k,\a} (S_{t,\a}),$ when $x<s_{t,\a}.$ If $x\in (-\infty,b)$  then $x-D\in  (-\infty,b)$ a.s.  In addition, if $x_n\to x$ then $x_n-D$ converges weakly to $x-D.$
Since the function $g_{t+1,\a}$ is bounded and continuous on $(-\infty,b),$ then the function $E[g_{t+1,\a}(x-D)]$ is continuous on $(-\infty,b).$  Since $b$ is arbitrary, this function is continuous on $\R.$ Formula \eqref{eqn:Gna} can be rewritten as
\[G_{t+1,\a} (x) = (1-\a) \bar{c}x + E[h(x-D)] + \a E[g_{t+1,\a} (x-D)] + \a\bar{c}\E[D],
\]
where all the summands are continuous functions.  In particular, $E[h(x-D)]$ is a nonnegative convex real-valued function on $\R,$ and therefore it is continuous. As shown above in this paragraph, the  function $E[g_{t+1,\a}(x-D)]$ is continuous too.  Thus, the function $G_{t+1,\a}$ is continuous.
%
%Let us fix $y\in\X.$ Define the following function
%\begin{align*}
%	g_{k+1,\a}(x) =
%	\begin{cases}
%		v_{k+1,\a}(x) + \bar{c}x  		& \text{if } x\leq y + 1 , \\
%		v_{k+1,\a}(y+1) + \bar{c}(y+1)  	& \text{if } x > y + 1 . \\
%	\end{cases}
%%	\label{eqn:g k+1}
%\end{align*}
%Since the functions $v_{k+1,\a} (x)$ and $\bar{c}x$ are continuous, then the function $g_{k+1,\a} (x)$ is continuous.
%In view of \eqref{eqn:vna}, the function $g_{k+1,\a} (x)$ is bounded on $\X,$ since $g_{k+1,\a} (x)$ is
%continuous on the interval $[S_{k,\a},\max\{S_{k,\a},y + 1\}]$ and $g_{k+1,\a} (x) = K + G(S_{k,\a})$ for
%all $x \leq S_{k,\a}.$ Therefore,
%\begin{align}
%	\begin{split}
%	& \lim_{z\to y} (1-\a) \bar{c}z + E[h(z-D)] + \a E[g_{k+1,\a} (z-D)]  \\
%	= & (1-\a) \bar{c}y + E[h(y-D)] + \a E[g_{k+1,\a} (y-D)],
%	\end{split}
%	\label{eqn:limit g k+1}
%\end{align}
%where the equality holds since the function $\bar{c}x$ is continuous, the function $E[h(x-D)]$ is convex
%on $\X$ and hence it is continuous, and $z-D$ converges weakly to $y-D$ as $z\to y$ and the function
%$g_{k+1,\a} (x)$ is continuous and bounded.
%
%Observe that $G_{k+1,\a} (x) = (1-\a) \bar{c}x + E[h(x-D)] + \a E[g_{k+1,\a} (x-D)] + \a\bar{c}\E[D]$
%for all $x\leq y + 1,$
%then \eqref{eqn:limit g k+1} implies that $\lim_{z\to y}G_{k+1,\a} (z) = G_{k+1,\a} (y).$ Therefore, the function
%$G_{k+1,\a} (x)$ is continuous.
%
Hence, the induction arguments imply that $v_{t,\a}(x)$ and $G_{t,\a} (x),$
$t=0,1,\ldots,$ are continuous functions.
\end{proof}

Consider an infinite-horizon inventory control problem.
If $(\sa,\Sa)$ policy is optimal, then the value function can be written as follows.
\begin{align}
\va (x) =
	\begin{cases}
		\Ga (x) - \bar{c}x  		& \text{if } x \geq \sa , \\
		K + \Ga (\Sa) - \bar{c}x  	& \text{if } x < \sa , \\
	\end{cases}
	\label{eqn:va sS}
\end{align}
where $\Ga(x)$ is defined in \eqref{eqn:Ga}.

The following theorem describes the continuity of value functions for infinite-horizon inventory control problems considered in this paper.
According to Feinberg et al. \cite[Theorem 2]{FEINBERG2012}, we know that $v_{N,\a}(x) \uparrow \va (x)$ as $N\to+\infty$ for all $x\in\X.$ We further prove that such  convergence is uniform, and the continuity of finite-horizon value functions implies the continuity of infinite-horizon value functions.

\begin{theorem}\label{thm:cont large df}
Consider an infinite-horizon inventory control problem with expected total discounted
cost criterion. The functions $\va(x)$ and $\Ga (x)$
are continuous on $\X$ for all $\a\in[0,1).$
\end{theorem}

\begin{proof}
Consider $\a^*$ defined in \eqref{def a*}. According to Theorem \ref{thm:general results},
if $\a\in[0,\a^*],$ then a policy that never orders is optimal for infinite-horizon problem,
which, in view of \eqref{eqn:va}, implies that $\va (x)= \Ga (x) - \bar{c}x.$
As follows from Lemma \ref{pro:1}(ii), the function $\Ga (x)$
is convex on $\X$ and then continuous. Therefore the function $\va(x)$ is also convex on $\X$
and hence continuous.

Consider a $N$-horizon optimal policy $\phi^N=(\phi^N_0,\phi^N_1,\ldots,\phi^N_{N-1})$ and
an infinite-horizon optimal policy $\phi.$ Define a policy $\psi=(\psi_0,\psi_1,\ldots)$ as
\begin{align*}
	\psi_t =
	\begin{cases}
		\phi^N_t  	& \text{if }  0\leq t \leq N-1 ,  \\
		\phi 		& \text{if }  t \geq N .
	\end{cases}
\end{align*}	
Then, for all $x\in\X,$
\begin{align}\label{eqn:ubd va-vna}
	v_{N,\a}(x) \leq \va (x) \leq \va ^{\psi} (x)
	= v_{N,\a} (x) + \E_x^{\psi}[\sum_{t=N}^{\infty}	\a^t c(x_t,a_t)] ,
\end{align}
where the first inequality holds because all costs are non-negative, the second inequality
is straightforward, and the last equality follows from
the definition of $\va ^{\psi} (x)$ and the optimality of $\phi^N.$

If $\a\in(\a^*,1),$ then $\phi$ can be chosen as the $(\sa,\Sa)$ policy; Theorem
\ref{thm:general results}.
Consider $N> N_{\a}$ whose existence is stated in Theorem \ref{thm:general results}(ii),
and $S_{t,\a}$ and $s_{t,\a}$ defined in \eqref{eqn:def S} and \eqref{eqn:def s} respectively
with $f(x):=G_{t,\a}(x),$ $x\in\X$ for all $t = N_{\a},\ldots,N-1.$
Then Theorem \ref{thm:general results}(ii) implies that $\phi^N$ can be chosen as the
policy that follows $(s_{N-t-1,\a},S_{N-t-1,\a})$ policy, $t = 0,\ldots,N-N_{\a}-1,$
and from then on never orders the inventory.

Let us fix $z\in\X.$ Consider $x_0\in (-\infty,z).$ According to Feinberg and Lewis
\cite[Theorem 6.11(ii)]{FEINBERG2015}, each sequence of pairs
$\{(s_{t,\a},S_{t,\a})\}_{t=N_{\a},N_{\a}+1,\ldots}$ is bounded. Thus there exists a constant $M_S$ such
that $S_{t,\a}\leq M_S$ for all $t=0,1,\ldots\ .$

Since $\psi_t = \phi^N_t$ for all $ t\leq N-1,$ then given $x_0 < z,$
\begin{align}\label{eqn:bd xT}
	s_{N_{\a}}-\bS_{N_{\a}}\leq x_N \leq \max\{x_0,S_{N_{\a}},\ldots,S_{N-1}\} \leq
	\max\{z,M_S\}, \qquad \textit{a.s.},
\end{align}
where the first two inequalities hold because $\phi^N$ is $(s,S)$ policy at first $N-N_{\a}$
steps and never orders at the remaining $N_{\a}$ steps, and the last inequality holds because
$x_0<z$ and $S_{t,\a}\leq M_S$ for all $t=0,1,\ldots\ .$

Since $\psi_t=\phi$ for all $t\geq N,$ then in view of \eqref{eqn:bd xT} and the definition
of $\phi,$ the following inequalities hold:
\begin{align}
	0\leq a_t \leq \max\{0,\Sa - x_t\}, \quad & t\geq N , \label{eqn:at} \\
	\sa - D \leq x_t \qquad\textit{a.s.},
	\quad & t > N , \text { and }  \label{eqn:xt} \\
	\sa\leq x_t + a_t\leq \max\{x_N,\Sa\} \leq \max\{z,M_S,\Sa \},
	\quad & t\geq N . \label{eqn:xt+at}
\end{align}

According to Theorem \ref{thm:general results}, $\sa$ and $\Sa$ are real numbers.
Therefore, \eqref{eqn:xt+at} implies that there exists a constant $M_1$ such that
for all $t\geq N,$
\begin{align}\label{eqn:ubd Eh}
	\E[h(x_t+a_t-D)]\leq \E[h(\sa-D)] + \E[h(\max\{z,M_S,\Sa \}-D)] \leq  M_1 ,
\end{align}
where the first inequality holds because the function $\E[h(x-D)]$ is non-negative and
convex on $\X,$ and the last one holds because $z,$ $M_S,$ and $\Sa$ are real numbers
and $\E[h(x-D)]<+\infty,$ $x\in\X.$

In addition, there exists a constant $M_2$ such that for all $t\geq N$ and $x_0<z,$
\begin{align}
\begin{split}
	0\leq \E_{x_0}^{\psi}[a_t]
	& \leq \E[\max\{ 0,\Sa-\sa+D,\Sa-s_{N_{\a}}+\bS_{N_{\a}+1} \}] \\
	&= \max\{0,\Sa-\sa+\E[D],\Sa-s_{N_{\a}}+(N_{\a}+1)\E[D] \} \leq M_2 ,
\end{split}\label{eqn:ubd Eat}
\end{align}
where the first two inequalities follows from \eqref{eqn:at}, \eqref{eqn:xt} and
\eqref{eqn:bd xT}, the equality is straightforward, and the last inequality holds because
all quantities are real numbers.

In view of \eqref{inventory-control:cost function}, \eqref{eqn:ubd Eh} and \eqref{eqn:ubd Eat},
for $M_3:=K+M_1+M_2$ such that for all $t\geq N$ and $x_0<z,$
\begin{align}
	\E_{x_0}^{\psi}[c(x_t,a_t)]\leq K+\bar{c}\E_{x_0}^{\psi}[a_t]+\E[h(x_t+a_t-D)]\leq M_3<+\infty .
	\label{eqn:ubd}
\end{align}

Since the cost $c(x,a)$ is non-negative for all $(x,a)\in \X\times\A,$ then according to
Feinberg et al. \cite[Theorem 2]{FEINBERG2012}, $v_{N,\a}(x) \uparrow \va (x)$ as
$N\to+\infty$ for all $x\in\X.$ Therefore,
\begin{align}
\begin{split}
	\sup_{x_0\in (-\infty,z)} |\va (x_0) - v_{N,\a} (x_0)| & \leq
	\E_{x_0}^{\psi}[\sum_{t=N}^{\infty}\a^t c(x_t,a_t)] , \\
	& \leq \sum_{t=N}^{\infty}\a^t M_3 = \frac{\a^N M_3 }{1-\a} \to 0, \text{ as } N\to+\infty,
\end{split}
\label{eqn:uniform convergence}
\end{align}
where the first inequality follows from \eqref{eqn:ubd va-vna}, the second one follows from
\eqref{eqn:ubd}, and the equality is straightforward. In view of
\eqref{eqn:uniform convergence}, the function $v_{N,\a} (x)$ converges uniformly
to the function $\va (x)$ on $(-\infty,z)$ as $N\to+\infty.$
Therefore, according to the uniform limit theorem,
since the function $v_{N,\a} (x)$ is continuous on $\X$ for all $N=1,2,\ldots$
(Theorem \ref{thm:cont finite-horizon}), then the function $\va (x)$ is continuous on
$(-\infty,z).$ Since $z$ can be chosen arbitrarily, thus the function $\va (x)$ is
continuous on $\X.$

Let us fix $y\in\X.$ Define the following function
\begin{align*}
	g_{\a}(x) =
	\begin{cases}
		v_{\a}(x) + \bar{c}x  		& \text{if } x\leq y + 1 , \\
		v_{\a}(y+1) + \bar{c}(y+1)  	& \text{if } x > y + 1 . \\
	\end{cases}
%	\label{eqn:g alpha}
\end{align*}
Since the functions $v_{\a} (x)$ and $\bar{c}x$ are continuous, then the function $g_{\a} (x)$ is continuous.
In view of \eqref{eqn:va}, the function $g_{\a} (x)$ is bounded on $\X.$ Therefore,
\begin{align}
	\begin{split}
	& \lim_{z\to y} \{ (1-\a) \bar{c}z + E[h(z-D)] + \a E[g_{\a} (z-D)] \} \\
	 = & (1-\a) \bar{c}y + E[h(y-D)] + \a E[g_{\a} (y-D)] ,
	\end{split}
	\label{eqn:limit g alpha}
\end{align}
where the equality holds since the function $\bar{c}x$ is continuous, the function $E[h(x-D)]$ is convex
on $\X$ and hence it is continuous, and $z-D$ converges weakly to $y-D$ as $z\to y$ and the function
$g_{\a} (x)$ is continuous and bounded.

Observe that $G_{\a} (x) = (1-\a) \bar{c}x + E[h(x-D)] + \a E[g_{\a} (x-D)] + \a\bar{c}\E[D]$
for all $x\leq y + 1,$
then \eqref{eqn:limit g alpha} implies that $\lim_{z\to y}G_{\a} (z) = G_{\a} (y).$ Therefore, the function
$G_{\a} (x)$ is continuous.
\end{proof}

Theorems~\ref{thm:cont finite-horizon} and \ref{thm:cont large df} imply the following corollary.

\begin{corollary}\label{thm:ordering at sa}

The statements of Theorems \ref{thm:sS policy cond holds}, \ref{thm:general results}, and \ref{thm:general results_ig} remain correct, if the second sentence of Definition~\ref{def:sS policy} is modified in the following way:
	a policy is called
	an $(s_t,S_t)$ policy at step $t,$ if it orders up to the level $S_t,$ if $x_t<s_t,$
	 does not order, if $x_t > s_t,$ and either does not order or orders up to the level $S_t,$ if $x_t=s_t.$
%
%	\end{enumerate}
\end{corollary}

\begin{proof} The proofs of Theorems \ref{thm:sS policy cond holds}, \ref{thm:general results}, and \ref{thm:general results_ig} are based on the fact that
$K+g(S)<g(x),$ if $x<s,$ and $K+g(S)\ge g(x),$ if $x\ge s,$ where $g=G_{t,\alpha},$ $S=S_{t,\a},$ and $s=s_{t,\a}$ for a finite-horizon problem and $g=G_{\alpha},$ $S=S_{\a},$ and $s=s_{\a}$ for the infinite-horizon problem.  Since the function $g$ is continuous in both cases, we have that $K+g(S)=g(s).$ Thus both actions are optimal at the state $s.$ \end{proof}

\begin{remark}\label{rm:corollary sS policy} 	Corollary \ref{thm:ordering at sa} also follows from the properties of the sets of optimal decisions  $A_{t,\a}(x) := \{ a\in\A:v_{t+1,\a}(x) = c(x,a) + \a E[v_{t,\a}(x + a - D)] \},$ $t=0,1,\ldots,$ for finite-horizon problems and $A_{\a}(x) := \{ a\in\A:v_\a(x) = c(x,a) + \a E[v_\a(x + a - D)] \}$ for infinite-horizon problem, $x\in\X,$ where the one-step cost function $c$ is defined in \eqref{inventory-control:cost function}. The solution multifunctions $A_{t,\a}(\cdot),$ $t=0,1,\ldots,$ and $A_{\a}(\cdot)$ are compact-valued (see  Feinberg et al.,~\cite[Theorem 2]{FEINBERG2012} or Feinberg and Lewis~\cite[Theorem 3.4] {FEINBERG2015}) and upper semi-continuous (this is true in view of Feinberg and Kasyanov~\cite[Statement B3]{FK15} because the value  functions are continuous and the optimality operators take infimums
of inf-compact functions).  Since upper semi-continuous, compact-valued set-valued functions are closed (Nikaido \cite[Lemma 4.4]{Nik68}), the graphs of the solution  multifunctions are closed.  Since $0\in \A_\alpha(x)$ for all $x>s,$ then $0\in A(s).$
%Thus, it is also optimal not to order inventory at the point $s.$
Similarly, $0\in A_{t,\a}(s_t),$ $t=0,1,\dots\ .$
\end{remark}

%	Corollary \ref{thm:ordering at sa} also follows from general MDPs theories. According to Feinberg et al. \cite[Lemma 3.3]{FKV14}, the cost function is $\mathbb{KN}$-inf-compact; see Definition 1.3 in Feinberg et al. \cite{FKV14}. Furthermore, according to Feinberg and Kasyanov \cite[Statement B3]{FK15} and Nikaido \cite[Lemma 4.4]{Nik68}, since the value function is continuous then the solution multifunctions
%$A_{t,\a}(x) := \{ a\in\A:v_{t+1,\a}(x) = c(x,a) + \a E[v_{t,\a}(x + a - D)] \},$ $x\in\X,$ for finite-horizon problem and
%$A_{\a}(x) := \{ a\in\A:v_\a(x) = c(x,a) + \a E[v_\a(x + a - D)] \},$ $x\in\X,$ for infinite-horizon problem have closed graphs.
%Therefore, the action $a=0$ is optimal for all $x>s_{t,\a}$ $(x>s_{\a})$ and the action $a = S_{t,\a} - x$ $(a = S_{\a} - x)$ is optimal for all $x<s_{t,\a}$ $(x<s_{\a})$ imply that both actions are optimal at the state $s_{t,\a}$ $(s_{\a}).$
%\end{remark}

\noindent
{\bf Acknowledgement.}
 This research  was partially supported by NSF grants CMMI-1335296 and CMMI-1636193.
The authors thank Jefferson Huang for  valuable comments.

%
%	bibliography
%

%
\end{document}